\documentclass[10pt]{amsart}

\newtheorem{thm}{Theorem}[section]

\newtheorem{obs}{Remark}[section]
\newtheorem{defin}{Definition}[section]

\usepackage{amssymb}
\usepackage{amsmath}
\usepackage{amsfonts}
\usepackage{graphicx}

\numberwithin{equation}{section}

\usepackage[T1]{fontenc}
\usepackage{palatino}

\begin{document}
\title[vortex stretching and criticality]
{Vortex stretching and criticality for the 3D NSE}
\author{R. Dascaliuc}
\address{Department of Mathematics\\
Oregon State University\\ Corvallis, OR 97331}
\author{Z. Gruji\'c}
\address{Department of Mathematics\\
University of Virginia\\ Charlottesville, VA 22904}
\dedicatory{Dedicated to Professor Peter Constantin on the occasion
     of his $60$th birthday, with admiration}
\date{\today}

\begin{abstract}
A mathematical evidence -- in a statistically significant sense --
of a \emph{geometric scenario} leading to \emph{criticality} of the
Navier-Stokes problem is presented.
\end{abstract}

\maketitle

\newpage

\section{Prologue}

\noindent 3D Navier-Stokes equations (NSE) -- describing a flow of
3D incompressible viscous fluid -- read
\[
 u_t+(u\cdot \nabla)u=-\nabla p + \triangle u,
\]
supplemented with the incompressibility condition $ \, \mbox{div} \,
u = 0$, where $u$ is the velocity of the fluid and $p$ is the
pressure (here, the viscosity is set to $1$); taking the curl yields
the vorticity formulation,
\[
 \omega_t+(u\cdot \nabla)\omega= (\omega \cdot \nabla)u + \triangle
 \omega,
\]
where $\omega = \, \mbox{curl} \, u$ is the vorticity of the fluid.

\medskip

It is well known that both globally \cite{Le34, H51} and
uniformly-locally (with suitable spatial decay at infinity) \cite{L-R02}
finite energy data generate global-in-time weak (distributional)
solutions to the 3D Navier-Stokes equations (NSE), satisfying global
and local energy inequality, respectively. Despite much effort --
since the pioneering work of Leray \cite{Le34} in 1930's -- the question of
whether weak solutions may exhibit finite-time singularities remains
an open problem. It is known that the set of all possible
singularities is small -- the one-dimensional (parabolic) Hausdorff
measure of the singular set in $\Omega \times (0,T)$ is zero for any
$T>0$ \cite{CKN82}; here, $\Omega$ is a global spatial domain.

\medskip

There are various regularity criteria preventing finite-time
formation of singularities, mainly expressed either as a local or a
global condition on a weak solution over a spatiotemporal domain, or
as a condition on a regular solution approaching a potential
singular time $T^*$. The conditions are given as boundedness of a
suitable spatiotemporal norm, the common trait being that the norm
is scaling-invariant (critical) with respect to the natural scaling
in the Navier-Stokes model. In contrast, the \emph{a priori} bounded
quantities are all subcritical; moreover, there is a \emph{scaling
gap} between a  regularity criterion in view and the corresponding
\emph{a priori} bounded quantity. As an illustration, here are two
classical examples -- for the velocity formulation -- in $L^p$ and
$BMO$ spaces. The regularity criteria are boundedness in $L^\infty_t
L^3_x$ \cite{ISS03} and $L^2_t BMO_x$ \cite{KT00}, and the corresponding \emph{a
priori} bounded quantities are $L^\infty_t L^2_x$ \cite{Le34, H51} and $L^1_t
L^\infty_x$ \cite{FGT81}, respectively. These are manifestations of
\emph{supercriticality} of the Navier-Stokes problem.

\medskip

A rigorous study of \emph{geometric depletion of the nonlinearity}
in the 3D NSE (as well as in the 3D Euler equations) was pioneered
by Constantin, and it was based on a singular integral
representation for $\alpha$ -- the stretching factor in the
evolution of the vorticity magnitude $|\omega|$ depleted by
\emph{local coherence of the vorticity direction} -- ``the story of
alpha and omega'' \cite{Co94}. This is fundamental as there is ample
evidence -- both numerical \cite{AKKG87, JWSR93, S81, VM94, SJO91} and
theoretical \cite{CPS95, GGH97, GFD99, Oh09} -- that regions of intense
vorticity tend to self-organize in \emph{coherent vortex
structures}, most notably, quasi one-dimensional vortex filaments,
displaying strong local coherence of the vorticity direction.

\medskip

The mechanism of the geometric depletion of the nonlinearity was
subsequently exploited in \cite{CoFe93} to show that as long as the regions
of intense vorticity exhibit local Lipschitz-coherence of the
vorticity direction, no finite-time blow up can occur, and later in
\cite{daVeigaBe02} where the Lipschitz-coherence was replaced by the
$\frac{1}{2}$-H\"older coherence. A spatiotemporal localization of
the $\frac{1}{2}$-H\"older coherence regularity criterion was
performed in \cite{GrZh06, Gr09}, and independently in \cite{ChKaLe07}. The
aforementioned regularity criteria are all pointwise
coherence conditions; hence, necessarily supercritical with respect to the NSE
scaling. A local, scaling-invariant (critical) criterion over a
parabolic cylinder below a potential singular point $(x_0,t_0)$,
\[
 \int_{t_0-(2R)^2}^{t_0} \int_{B(x_0,2R)} |\omega(x,t)|^2
 \rho_{\frac{1}{2}, 2R}^2 (x,t) \, dx \, dt < \infty
\]
where
\[
 \rho_{\gamma, r}(x,t) = \sup_{y \in B(x,r), y \neq x} \frac{\Bigl|\sin
 \varphi \Bigl(\eta(x,t), \eta(y,t)\Bigr)\Bigr|}{|x-y|^\gamma}
\]
is a $\gamma$-H\"older measure of coherence of the vorticity
direction $\eta$ at the point $(x,t)$, was presented in \cite{GrGu10-1}. On
the other hand, a corresponding (subcritical) \emph{a priori} bound
had been previously obtained in \cite{Co90},
\[
 \int_0^T \int_{\mathbb{R}^3} |\omega(x,t)| |\nabla \eta (x,t)|^2 \,
 dx \, dt \le \frac{1}{2} \nu^{-2} \int_{\mathbb{R}^3} |u_0(x)|^2 \,
 dx
\]
where $\nu$ is the viscosity.

\medskip

A different geometric approach to the study of possible singularity
formation in 2D and 3D incompressible flows was developed by Cordoba
and Fefferman \cite{CF01-1, CF01-2, CF02}; in particular, non-existence of `tube
collapse singularities' in 3D incompressible inviscid flows was
shown in \cite{CF01-2}, and non-existence of a more general class of
`squirt singularities' in incompressible flows -- including the
flows described by the 3D NSE -- was presented in \cite{CFD04}.

\medskip

The purpose of this Note is to present a mathematical evidence -- in
a statistically significant sense -- of a geometric scenario leading
to \emph{criticality} of the Navier-Stokes problem. More
precisely, utilizing the \emph{ensemble averaging process}
introduced in our recent study of turbulent cascades in
\emph{physical scales} of 3D incompressible fluid flows
\cite{DaGr11-1, DaGr11-2, DaGr12-1, DaGr12-2}, we
show that the ensemble-averaged vortex stretching term is positive
across a range of scales extending from a square root of a
Kraichnan-type micro-scale to the macro-scale. Combining this with the
\emph{a priori} bound on the decrease of the distribution function
of the vorticity obtained by Constantin in \cite{Co90} -- as well as with
the general mechanism of creation and dynamics of vortex
filaments in turbulent flows (cf. \cite{CPS95}) -- indicates a
geometric scenario in which the region of intense vorticity (defined
as the region in which the vorticity magnitude -- near a possible
singular time -- exceeds a fraction of the $L^\infty$-norm) comprises
of macro-scale-long vortex filaments with the diameters of the
cross-sections scaling like
$\displaystyle{\frac{1}{\|\omega(t)\|_\infty^\frac{1}{2}}}$. This is
exactly the \emph{scale} of \emph{local one-dimensional sparseness}
of the region of intense vorticity needed to prevent a formation of
a finite-time singularity \cite{Gr12}.

\section{Geometric measure-type regularity criterion}

In this section, we briefly recall a \emph{geometric measure-type}
regularity criterion for solutions to the 3D NSE obtained
recently by one of the authors; for details, see \cite{Gr12}.

\medskip

\begin{defin}
Let $x_0$ be a point in $\mathbb{R}^3$, $r>0$, $S$ an open subset of
$\mathbb{R}^3$ and $\delta$ in $(0,1)$.

\medskip

The set $S$ is \emph{linearly $\delta$-sparse around $x_0$ at scale
$r$ in weak sense} if there exists a unit vector $d$ in $S^2$ such
that
\[
 \frac{|S \cap (x_0-rd, x_0+rd)|}{2r} \le \delta.
\]
\end{defin}

\medskip

For $M>0$, denote by $\Omega_t(M)$ the vorticity super-level set at
time $t$; more precisely,
\[
 \Omega_t(M) = \{x \in \mathbb{R}^3: |\omega(x,t)| > M\}.
\]
The vorticity version of the local one-dimensional (linear)
sparseness regularity criterion is as follows.

\medskip

\begin{thm}\label{sparse_omega}
Suppose that a solution $u$ is regular on an interval $(0,T^*)$.

\medskip

Fix $\delta$ in $(0,1)$, and let
$h=h(\delta)=\frac{2}{\pi}\arcsin\frac{1-\delta^2}{1+\delta^2}$ and
$\alpha=\alpha(\delta)\ge\frac{1-h}{h}$. Assume that there exists
$\epsilon >0$ such that for any $t$ in $(T^*-\epsilon, T^*)$, either

\medskip

(i) \ $\displaystyle{t+\frac{1}{d_0^2 \|\omega(t)\|_\infty} \ge
T^*}$ ($d_0$ is an absolute constant appearing in the local-in-time
analytic smoothing in $L^\infty$; cf. \cite{Gr12}), or

\medskip

(ii) \ there exists $s=s(t)$ in $\Bigl[t+\frac{1}{4d_0^2
\|\omega(t)\|_\infty}, t+\frac{1}{d_0^2 \|\omega(t)\|_\infty}\Bigr]$
such that for any spatial point $x_0$, there exists a scale $r$,
$0<r\le \frac{1}{2d_0^2 \|\omega(t)\|^\frac{1}{2}_\infty}$, with the
property that the super-level set $\Omega_s(M)$ is linearly
$\delta$-sparse around $x_0$ at scale $r$ in weak sense; here,
$M=M(\delta)=\frac{1}{d_0^\alpha} \|\omega(t)\|_\infty$.

\medskip

Then, there exists $\gamma >0$ such that $\omega$ is in
$L^\infty\Bigl((T^*-\epsilon, T^*+\gamma); L^\infty\Bigr)$, i.e.,
$T^*$ is not a singular time.
\end{thm}

\medskip

The proof is based on a very intimate interplay between the
diffusion in the model -- represented by the local-in-time analytic
smoothing in $L^\infty$ -- and the geometric properties of the
harmonic measure (via the \emph{harmonic measure majorization
principle}).

\medskip

The analyticity estimate on solutions needed is a vorticity version
of the estimate given in \cite{Gu10}; this was based on a general
method for estimating uniform radius of spatial analyticity in
$L^p$-spaces introduced in \cite{GrKu98}, which was in turn inspired
by the (analytic) Gevrey-class method presented in \cite{FT89} (see
also \cite{FeTi98}).

\medskip

The key geometric harmonic measure estimate used in the proof is a
generalization of the classical Beurling's problem \cite{Beu33},
conjectured in \cite{Seg88} and solved by Solynin in \cite{Sol99} (a
symmetric version of the problem was previously resolved in
\cite{EssHa89}); for more details, see \cite{Gr12}.

\medskip

\begin{obs}
\emph{A rudimentary version of Theorem \ref{sparse_omega}
was previously obtained
in \cite{Gr01}. The condition needed in \cite{Gr01} is a much
stronger condition; essentially, a requirement of a local existence
of a sparse \emph{coordinate projection}. In contrast, all that is
needed here is a local sparseness of an one-dimensional \emph{trace}
in a \emph{very weak sense}.}
\end{obs}

\section{The region of intense vorticity}

There is strong numerical evidence that the regions of high
vorticity organize in coherent vortex structures \cite{S81, AKKG87,
SJO91, JWSR93, VM94} and in particular, in elongated
vortex filaments (tubes). In addition, an in-depth analysis of
creation and dynamics of vortex tubes in 3D turbulent
incompressible flows was
presented in \cite{CPS95} (see also \cite{GGH97, GFD99, Oh09}).

\medskip

Consider a flow near the first (possible) singular time $T^*$, and
define \emph{the region of intense vorticity} at time $t < T^*$ to be
the region in which the vorticity magnitude exceeds a fraction
of $\|\omega(t)\|_\infty$; keeping the notation from the previous
section, this corresponds to the set
$\displaystyle{\Omega_t\Bigl(\frac{1}{c_1}
\|\omega(t)\|_\infty}\Bigr)$, for some $c_1>1$.

\medskip

Denote a suitable \emph{macro-scale} associated with the flow by $R_0$.
The picture painted by numerical simulations indicates that
the region of intense vorticity comprises -- in a statistically significant
sense -- of vortex filaments with the lengths comparable to $R_0$.

\medskip

Let us for a moment accept this as a probable geometric blow up
scenario. The length scale associated with the diameters of the
cross-sections can then be estimated \emph{indirectly}, by
estimating the rate of the decrease of the total volume of the region
of intense vorticity $\displaystyle{\Omega_t\Bigl(\frac{1}{c_1}
\|\omega(t)\|_\infty}\Bigr)$.

\medskip

Taking the initial vorticity to be a finite Radon measure,
Constantin showed \cite{Co90} that the $L^1$-norm of the vorticity
is \emph{a priori} bounded over \emph{any} finite time-interval; a
desired estimate on the total volume of the region of intense
vorticity follows simply from the Tchebyshev inequality,
\[
 \, \mbox{Vol} \, \biggl( \Omega_t \Bigl(\frac{1}{c_1}
\|\omega(t)\|_\infty \Bigr) \biggr) \le \frac {c_2}{\|\omega(t)\|_\infty}
\ \ (c_2>1).
\]
This implies the decrease of the diameters of the cross-section of
at least $\displaystyle{\frac {c_3}{\|\omega(t)\|^\frac{1}{2}_\infty}}$
$(c_3>1)$,
which is exactly the scale of \emph{local one-dimensional sparseness}
of the region of intense vorticity needed to prevent the formation of singularities
presented in Theorem \ref{sparse_omega}. In other words, the
Navier-Stokes problem in this scenario becomes \emph{critical}.

\medskip

A key step in justifying this scenario is providing a
\emph{mathematical evidence} of persistence -- in a statistically
significant sense -- of the $R_0$-long vortex filaments (at this
point, the evidence is purely numerical). A term responsible for the
creation of vortex filaments is the \emph{vortex-stretching term},
\[
 (\omega \cdot \nabla)u \cdot \omega = S \omega \cdot \omega,
\]
where $S$ is the strain matrix. One way to identify the range of
(longitudinal) scales at which the dynamics of creation and
persistence of vortex filaments takes place is to identify the
\emph{range of scales of positivity} of $S \omega \cdot \omega$. In
the following section, we will show that the range of positivity of
$S \omega \cdot \omega$ -- in a statistically significant sense --
extends from a power of a Kraichnan-type micro-scale to
the macro-scale $R_0$. It is worth
pointing out that the argument is  \emph{dynamic} -- based on
ensemble averaging local dynamics described by the full 3D
Navier-Stokes system.

\section{A dynamic estimate on the vortex-stretching term across a range of scales}

We begin by recalling the concept of \emph{ensemble averaging} with
respect to $(K_1,K_2)$\emph{-covers at scale} $R$, introduced in our
work on existence and locality of turbulent cascades in
\emph{physical scales} of 3D incompressible flows \cite{DaGr11-1,
DaGr11-2, DaGr12-1, DaGr12-2} (for more details, see, e.g.,
\cite{DaGr12-2}).

\medskip

Let $R_0 > 0$, and assume (for convenience) that the macro-scale
domain of interest is the ball $B(0,R_0)$, $B(0,2R_0)$ contained in
the global spatial domain $\Omega$. Consider a locally integrable
physical density of interest $f$, and let $0 < R \le R_0$; the time
interval of interest is $(0,T)$.

\medskip

In what follows, we utilize refined spatiotemporal cut-off functions
$\phi=\phi_{x_0,R,T}=\psi\,\eta$, where $\eta=\eta_T(t)\in C^\infty
(0,T)$ and $\psi=\psi_{x_0,R}(x)\in\mathcal{D}(B(x_0,2R))$
satisfying
\begin{equation}\label{eta_def}
0\le\eta\le1,\quad\eta=0\ \mbox{on}\ (0,T/3),\quad\eta=1\ \mbox{on}\
(2T/3,T),\quad\frac{|\eta'|}{\eta^{\rho_1}}\le\frac{C_0}{T}\;
\end{equation}
and
\begin{equation}\label{psi_def}
0\le\psi\le 1,\quad\psi=1\ \mbox{on}\ B(x_0,R),
\quad\frac{|\nabla\psi|}{\psi^{\rho_2}}\le\frac{C_0}{R}, \quad
\frac{|\triangle\psi|}{\psi^{2\rho_2-1}}\le\frac{C_0}{R^2}\;,
\end{equation}
for some $\frac{1}{2} <\rho_1,\rho_2 < 1$. In particular,
$\phi_0=\psi_0 \eta$ where $\psi_0$ is the spatial cut-off (as
above) corresponding to $x_0=0$ and $R=R_0$.

\medskip

For $x_0$ near the boundary of the macro-scale domain, $S(0,R_0)$,
we assume additional conditions,
\begin{equation}\label{psi_bd}
0\le\psi\le\psi_0
\end{equation}
and, if  $B(x_0,R)\not\subset B(0,R_0)$, then
$\psi\in\mathcal{D}(B(0,2R_0))$ with $\psi=1\ \mbox{on}\ B(x_0,R)
\cap B(0,R_0)$ satisfying, in addition to (\ref{psi_def}), the
following:
\begin{equation}\label{psi_def_add1}
\begin{aligned}
&
\psi=\psi_0\ \mbox{on the part of the cone centered at zero and passing through}\\
& S(0,R_0)\cap B(x_0,R)\ \mbox{between}\  S(0,R_0)\ \mbox{and}\
S(0,2R_0)
\end{aligned}
\end{equation}
and
\begin{equation}\label{psi_def_add2}
\begin{aligned}
&
\psi=0\ \mbox{on}\ B(0,R_0)\setminus B(x_0,2R)\ \mbox{and outside the part of the cone}\\
 &
 \mbox{centered at zero and passing through}\ S(0,R_0)\cap B(x_0,2R)\\
 &
 \mbox{between}\  S(0,R_0)\ \mbox{and}\ S(0,2R_0).
\end{aligned}
\end{equation}

\medskip

A \emph{physical scale $R$} is realized via suitable ensemble
averaging of the localized quantities with respect to
`$(K_1,K_2)$-covers' at scale $R$.

\medskip

Let $K_1$ and $K_2$ be two positive integers, and $0 < R \le R_0$. A
cover $\{B(x_i,R)\}_{i=1}^n$ of the macro-scale domain $B(0,R_0)$ is
a \emph{$(K_1,K_2)$-cover at scale $R$} if
\[
 \biggl(\frac{R_0}{R}\biggr)^3 \le n \le K_1
 \biggr(\frac{R_0}{R}\biggr)^3,
\]
and any point $x$ in $B(0,R_0)$ is covered by at most $K_2$ balls
$B(x_i,2R)$. The parameters $K_1$ and $K_2$ represent the maximal
\emph{global} and \emph{local multiplicities}, respectively.
Considering the time-averaged, per unit mass -- spatially localized
to the cover elements $B(x_i, R)$ -- local quantities
$\hat{f}_{x_i,R}$,
\[
\hat{f}_{x_i,R} = \frac{1}{T} \int_0^T \frac{1}{R^3}
\int_{B(x_i,2R)} f(x,t) \phi^\delta_{x_i,R,T} (x,t) \, dx \, dt
\]
(for some $0 < \delta \le 1)$, the \emph{ensemble average} $\langle
F\rangle_R$ is defined as
\[
 \langle F\rangle_R = \frac{1}{n} \sum_{i=1}^n
 \hat{f}_{x_i,R}\,.
\]

\medskip

The ensemble averages (with the fixed multiplicities $K_1$ and
$K_2$) act as a `detector' of significant sign-fluctuations of the
density in view. More precisely, if the density exhibits significant
sign-fluctuations on the scales comparable or greater than $R$, the
ensemble averages at scale $R$ -- with respect to all admissible
$(K_1,K_2)$-covers -- will respond by exhibiting a wide range of
values, from positive through zero to negative. This can be seen by
rearranging the cover elements to emphasize the positive and the
negative parts of the density, respectively. The larger the
multiplicities, the finer the detection. In contrast, if the
ensemble averages at scale $R$ -- with respect to all admissible
$(K_2,K_2)$-covers (again, with the fixed multiplicities) -- are
nearly independent on the particular choice of the cover, and say
positive, this indicates that the density is essentially positive on
the scales comparable or greater than $R$.

\medskip

As expected, for a non-negative density $f$, all the averages are
comparable to each other throughout the full range of scales $R$, $0
< R \le R_0$; in particular, they are all comparable to the simple
average over the integral domain. More precisely,
\begin{equation}\label{k*}
 \frac{1}{K_*} F_0 \le \langle F \rangle_R \le K_* F_0
\end{equation}
for all $0 < R \le R_0$, where
\[
 F_0=\frac{1}{T}\int \frac{1}{R_0^3} \int  f(x,t)
 \phi_0^\delta (x,t) \, dx \, dt,
\]
and $K_* = K_*(K_1,K_2) > 1$.

\bigskip

Consider now a global-in-time weak solution $u$ (say, a
global-in-time `local Leray solution' on $\mathbb{R}^3 \times
(0,\infty)$ in the sense of \cite{L-R02}), and let $T$ be the first
(possible) singular time.

\medskip

A spatiotemporal localization of the evolution of the enstrophy was
presented in \cite{GrZh06, Gr09}. Considering a $(K_1,K_2)$-cover
$\{B(x_i,R)\}_{i=1}^n$ at scale $R$, the following expression for
the time-integrated $B(x_i,R)$-localized vortex-stretching terms
transpires,

\begin{align}\label{loc}
 \int_0^t \int (\omega \cdot \nabla)u \cdot \phi_i \, \omega \; dx
 \; ds
 &=
 \int \frac{1}{2}|\omega(x,t)|^2\psi_i(x) \; dx + \int_0^t \int
 |\nabla\omega|^2\phi_i
 \;  dx \; ds\notag\\
 &- \int_0^t \int \frac{1}{2}|\omega|^2 \bigl((\phi_i)_s+\triangle\phi_i\bigr) \; dx
 \; ds\notag\\
 &- \int_0^t \int \frac{1}{2}|\omega|^2 (u \cdot \nabla\phi_i) \; dx
 \; ds,
\end{align}
for any $t$ in $(2T/3,T)$, and $1 \le i \le n$.

\medskip

Denoting the time-averaged local vortex-stretching terms per unit
mass associated to the cover element $B(x_i,R)$ by $VST_{x_i,R,t}$,
\begin{equation}\label{locfluxiav}
VST_{x_i,R,t} = \frac{1}{t} \int_0^t \frac{1}{R^3} \int (\omega
\cdot \nabla)u \cdot \phi_i \, \omega \, dx \, ds,
\end{equation}
the main quantity of interest is the ensemble average of
$\{VST_{x_i,R,t}\}_{i=1}^n$,
\begin{equation}\label{PhiR}
 \langle VST\rangle_{R,t} = \frac{1}{n}\sum_{i=1}^n VST_{x_i,R,t}.
\end{equation}

\medskip

Before stating the theorem, let us introduce the key macro-scale
quantities, $E_0$, $P_0$ and $\sigma_0$. Denote by $E_{0,t}$
time-averaged enstrophy per unit mass associated with the
macro-scale domain $B(0,2R_0) \times (0,t)$,
\[
 E_{0,t}=\frac{1}{t}\int_0^t \frac{1}{R_0^3} \int \frac{1}{2}|\omega|^2
 \phi_0^{1/2} \, dx \, ds,
\]
by $P_{0,t}$ a modified time-averaged palinstrophy per unit mass,
\[
 P_{0,t}= \frac{1}{t}\int_0^t \frac{1}{R_0^3} \int |\nabla\omega|^2
 \phi_0 \, dx \, ds
 + \frac{1}{t}\frac{1}{R_0^3} \int \frac{1}{2}|\omega(x,t)|^2
 \psi_0(x) \, dx
\]
(the modification is due to the shape of the temporal cut-off
$\eta$), and by $\sigma_{0,t}$ a corresponding Kraichnan-type scale,
\[
 \sigma_{0,t}=\biggl(\frac{E_{0,t}}{P_{0,t}}\biggr)^\frac{1}{2}.
\]

\medskip

Until now, there was no connection between the spatial macro-scale
$R_0$ and the global time scale $T$. At this point, it is convenient
to assume $R_0 \le \sqrt{T}$ (in addition, without loss of generality,
suppose that $T \le 1$); in the case $R_0 > \sqrt{T}$, the
proof can be modified similarly to the calculations in
\cite{DaGr11-2, DaGr12-1}.

\medskip

\begin{thm}
Let $u$ be a global-in-time local Leray solution on $\mathbb{R}^3
\times (0,\infty)$, regular on $(0,T)$. Suppose that, for some
$t \in (2T/3, T)$,
\begin{equation}\label{cond}
C \max\{M_0^\frac{1}{2}, R_0^\frac{1}{2}\} \,
\sigma_{0,t}^\frac{1}{2} < R_0
\end{equation}
where $\displaystyle{M_0=\sup_t \int_{B(0,2R_0)}
|u|^2 < \infty}$, and $C > 1$ a suitable constant depending only on the
cover parameters.

Then,
\begin{equation}\label{estimate}
 \frac{1}{C} \, P_{0,t} \le \langle VST \rangle_{R,t} \le
 C \, P_{0,t}
\end{equation}
for all $R$ satisfying
\begin{equation}\label{range}
C \max\{M_0^\frac{1}{2}, R_0^\frac{1}{2}\} \,
\sigma_{0,t}^\frac{1}{2} \le R \le R_0.
\end{equation}
\end{thm}

\begin{obs}
\emph{The macro-scale domain $B(0,R_0)$ is placed at the origin for
convenience only; it can be placed anywhere in $\mathbb{R}^3$.}
\end{obs}

\begin{proof}
Recall that
\begin{align}\label{locc}
 \int_0^t \int (\omega \cdot \nabla)u \cdot \phi_i \, \omega \; dx
 \; ds
 &=
 \int \frac{1}{2}|\omega(x,t)|^2\psi_i(x) \; dx + \int_0^t \int
 |\nabla\omega|^2\phi_i
 \;  dx \; ds\notag\\
 &- \int_0^t \int \frac{1}{2}|\omega|^2 \bigl((\phi_i)_s+\triangle\phi_i\bigr) \; dx
 \; ds\notag\\
 &- \int_0^t \int \frac{1}{2}|\omega|^2 (u \cdot \nabla\phi_i) \; dx
 \; ds,
\end{align}
for any $t$ in $(2T/3,T)$, and $1 \le i \le n$; the last two terms
need to be estimated.

\medskip

For the first term, the properties of the spatiotemporal cut-off
function $\phi_i$ --setting $\rho_1=\rho_2=3/4$ -- paired with
the condition $t > \frac{2}{3}T \ge
\frac{2}{3} R_0^2 \ge \frac{2}{3} R^2$ yield

\begin{equation}\label{1}
\int_0^t \int \frac{1}{2}|\omega|^2
\bigl((\phi_i)_s+\triangle\phi_i\bigr) \; dx \; ds \le C
\frac{1}{R^2} \int_0^t \int |\omega|^2 \phi_i^{1/2} \; dx \; ds.
\end{equation}

\medskip

The second term -- the localized transport term -- will be estimated
similarly as in \cite{GrZh06}; the powers of the cut-off
function $\phi_i$ will be distributed somewhat differently leading
to a bit more precise estimate.

\medskip

Setting the cut-off parameters $\rho_1$ and $\rho_2$ to $7/8$, the
following sequence of bounds transpires.

\begin{align}\label{3}
\int_0^t \int \frac{1}{2}|\omega|^2 & (u \cdot \nabla\phi_i) \; dx \; ds\notag\\
 & \le C \frac{1}{R} \int_0^t \int \Bigl(
 |\omega|^2 \phi_i \Bigr)^{3/4} |u| \ \Bigl(
 |\omega|^2 \phi_i^{1/2}\Bigr)^{1/4} \; dx \; ds\notag\\
 &\le C \frac{1}{R} \int_0^t \Bigl(\int |u|^{4/3} |\omega|^2 \phi_i \; dx\Bigr)^{3/4} \
 \ \Bigl(\int
 |\omega|^2 \phi_i^{1/2}\; dx\Bigr)^{1/4} \, ds\notag.\\
\end{align}

\medskip

The first spatial integral is bounded as follows,

\begin{align}\label{4}
\int |u|^{4/3} |\omega|^2 \phi_i & \; dx\notag\\
& \le \Bigl(\sup_s \int_{B(x_i,2R)} |u|^2 \, dx\Bigr)^{2/3} \
\Bigl(\int \bigl(|\omega|
\phi_i^{1/2}\bigr)^6 \, dx\Bigr)^{1/3}\notag\\
& \le C \ \Bigl(\sup_s \int_{B(x_i,2R)} |u|^2 \, dx\Bigr)^{2/3} \
\Bigl(\int |\nabla \bigl( \phi_i^{1/2} \omega \bigr)|^2 \, dx\Bigr)
\end{align}
(the last line by the Sobolev Embedding Theorem).

\medskip

Combining the bounds (\ref{3}) and (\ref{4}) leads to

\begin{align}\label{5}
\int_0^t \int & \frac{1}{2}|\omega|^2  (u \cdot \nabla\phi_i) \; dx \; ds\notag\\
 & \le C  \frac{1}{R}  \Bigl(\sup_s \int_{B(x_i,2R)} |u|^2 \,
 dx\Bigr)^{1/2}  \Bigl(\int_0^t \int |\nabla \bigl( \phi_i^{1/2} \omega
 \bigr)|^2 \, dx \, ds\Bigl)^{3/4}  \Bigl(\int_0^t \int
 |\omega|^2 \phi_i^{1/2} \, dx \, ds\Bigr)^{1/4}\notag\\
 & \le \frac{1}{8} \int_0^t \int |\nabla \bigl( \phi_i^{1/2} \omega
 \bigr)|^2 \, dx \, ds + C \frac{\Bigl(\sup_s \int_{B(x_i,2R)} |u|^2 \,
 dx\Bigr)^2}{R^2} \frac{1}{R^2} \int_0^t \int
 |\omega|^2 \phi_i^{1/2} \, dx \, ds.
\end{align}

\medskip

Utilizing the commutator estimate (with $\rho_1=\rho_2=3/4$)

\begin{align}\label{6}
\int |\nabla(\phi_i^\frac{1}{2} & \omega)|^2 \, dx\notag\\
& \le 2 \int |\nabla\omega|^2 \phi_i \, dx + C \int
\biggl(\frac{|\nabla\phi_i|}{\phi_i^\frac{1}{2}}\biggr)^2 \,
|\omega|^2 \, dx\notag\\
& \le 2 \int |\nabla\omega|^2 \phi_i \, dx + C \frac{1}{R^2}
\int |\omega|^2 \phi_i^{1/2} \, dx\notag\\
\end{align}
in the first term of the above inequality yields the final bound for
the localized transport term,

\medskip

\begin{align}\label{7}
\int_0^t \int & \frac{1}{2}|\omega|^2  (u \cdot \nabla\phi_i) \; dx \; ds\notag\\
& \le \frac{1}{4} \int_0^t \int |\nabla\omega|^2 \phi_i \, dx \, ds
+ C \frac{1}{R^2}
\int_0^t \int |\omega|^2 \phi_i^{1/2} \, dx \, ds\notag\\
& + C \ \biggl(\frac{\sup_s \int_{B(x_i,2R)} |u|^2 \,
 dx}{R}\biggr)^2 \frac{1}{R^2} \int_0^t \int
 |\omega|^2 \phi_i^{1/2} \, dx \, ds.
\end{align}
Note that the factor
\[
\frac{\sup_s \int_{B(x_i,2R)} |u|^2 \, dx}{R}
\]
is scaling-invariant, and -- consequently -- the bound (\ref{7}) is
dimensionally correct. However, for an arbitrary global-in-time
local Leray solution it is not \emph{a priori} bounded (it is
\emph{a priori} bounded, e.g., assuming a uniform-in-time bound on
the $L^3$-norm; this, however, automatically implies regularity
\cite{ISS03}). The best one can do in general is to simply write
$\displaystyle{\sup_s \int_{B(x_i,2R)} |u|^2 \, dx \le M_0}$.

\medskip

Taking this into account, and applying the bounds (\ref{1}) and
(\ref{7}) in the expression (\ref{locc}) -- describing the dynamics of
the vortex-stretching term localized to the cover element
$B(x_i,R)$ -- leads to

\begin{equation}\label{loccc}
 \int_0^t \int (\omega \cdot \nabla)u \cdot \phi_i \, \omega \; dx
 \; ds =
 \int \frac{1}{2}|\omega(x,t)|^2\psi_i(x) \; dx + \int_0^t \int
 |\nabla\omega|^2\phi_i
 \;  dx \; ds \, + I_i,
\end{equation}
where
\[
|I_i| \le  \frac{1}{4} \int_0^t \int |\nabla\omega|^2 \phi_i \, dx
\, ds + C \max\{M_0^2, R_0^2\} \frac{1}{R^4} \int_0^t \int
|\omega|^2 \phi_i^{1/2} \, dx \, ds.
\]

\medskip

Ensemble-averaging (\ref{loccc}) and utilizing the inequality (\ref{k*})
several times implies that as long as
\begin{equation}\label{ir}
C \max\{M_0^\frac{1}{2}, R_0^\frac{1}{2}\} \,
\sigma_{0,t}^\frac{1}{2} < R_0,
\end{equation}
\[
 \frac{1}{C} \, P_{0,t} \le \langle VST \rangle_{R,t} \le
 C \, P_{0,t}
\]
for all $R$ satisfying
\[
C \max\{M_0^\frac{1}{2}, R_0^\frac{1}{2}\} \,
\sigma_{0,t}^\frac{1}{2} \le R \le R_0.
\]
\end{proof}

\medskip

\begin{obs}
\emph{Suppose that $T$ is the first (possible) singular time, and that the macro-scale
domain contains some of the spatial singularities (at time $T$). This, paired with
the assumption that $u$ is a global-in-time local Leray solution
implies}
\[
 \lim_{t \to T^-} \sigma_{0,t} = 0;
\]
\emph{hence, the condition (\ref{cond}) in the theorem is automatically
satisfied for any $t$ near the singular time $T$.}
\end{obs}

\bigskip

\bigskip

\noindent ACKNOWLEDGMENTS \ The authors express their gratitude to
Professor Peter Constantin for being an invariable source of
mathematical inspiration, as well as for all his support over the
years. R.D. and Z.G. acknowledge the support of the National Science
Foundation via the grants DMS-1211413 and DMS-1212023, respectively;
Z.G. acknowledges the support of the Research Council of Norway via
the grant 213473-FRINATEK.

\end{document}